\documentclass[11pt,a4paper]{article}
\usepackage[T1]{fontenc}
\usepackage{lmodern}
\usepackage[margin=25mm]{geometry}
\usepackage{amsmath,amssymb,amsthm,mathtools}
\usepackage{xcolor}
\definecolor{linkblue}{RGB}{0,76,153}
\usepackage{microtype}
\usepackage{enumitem}
\usepackage{needspace}
\usepackage{xurl}
\usepackage[hidelinks]{hyperref}
\hypersetup{pdftitle={Periodic Joint Co-tiles of Independent Tiles},
 pdfauthor={Hu Tan and Ying Zhang}}
\allowdisplaybreaks[2]
\newtheorem{theorem}{Theorem}[section]
\newtheorem{lemma}[theorem]{Lemma}
\newtheorem{proposition}[theorem]{Proposition}
\newtheorem{corollary}[theorem]{Corollary}
\theoremstyle{definition}
\newtheorem{definition}[theorem]{Definition}
\newtheorem{example}[theorem]{Example}
\newtheorem{remark}[theorem]{Remark}
\newcommand{\Z}{\mathbb Z}
\newcommand{\Q}{\mathbb Q}
\newcommand{\R}{\mathbb R}

\newcommand{\ind}{\mathbf1}
\newcommand{\supp}{\operatorname{supp}}
\newcommand{\spanq}{\operatorname{span}_{\Q}}
\newcommand{\orb}{\operatorname{Orb}}

\numberwithin{equation}{section}
\title{Periodic Joint Co-tiles of Independent Tiles}
\author{\begin{tabular}[t]{c}
Hu Tan\textsuperscript{1}\\[3pt]
\small\textsuperscript{1}Academy of Mathematics and Systems Science\\
\small Chinese Academy of Sciences\\
\small Beijing 100190, China\\
\small\texttt{tanhu2020@amss.ac.cn}
\end{tabular}
\and
\begin{tabular}[t]{c}
Ying Zhang\textsuperscript{2}\\[3pt]
\small\textsuperscript{2}School of Mathematical Sciences\\
\small Soochow University\\
\small Suzhou 215006, China\\
\small\texttt{yzhang@suda.edu.cn}
\end{tabular}}
\date{}
\begin{document}
\maketitle
\begin{abstract}
For every $d\geq2$, we prove that $d-1$ independent finite tiles in $\Z^d$ admit a fully periodic joint co-tile whenever they admit a joint co-tile. Here independence means that the tiles contain the origin and every choice of one nonzero vector from each tile is linearly independent. This removes property $(\star)$ from the periodic existence theorem of Meyerovitch, Sanadhya and Solomon and recovers Bhattacharya's theorem when $d=2$. More generally, we prove periodic existence for any finite family admitting a joint co-tile whose indicator is a fully periodic real function plus finitely many bounded real functions, each invariant under a subgroup of rank $d-1$. Extending the Boolean normal-form method of Greenfeld and Tao, we obtain a configuration in the translation orbit closure whose restrictions to cosets of a finite-index lattice are either invariant under a rank-$(d-1)$ subgroup or given by a three-term affine fractional-part formula with density one half. The tiling equations pair the restrictions given by this formula and determine a finite bipartite graph. A two-coloring replaces those restrictions by constants while preserving every tiling equation. We deduce a characterization of periodic tilability by independent companions and a decision algorithm for families containing an independent $(d-1)$-subfamily, with the dimension as part of the input.
\end{abstract}

\section{Introduction and main results}\label{sec:intro}
Let $F$ be a nonempty finite subset of $\Z^d$. A set $A\subset\Z^d$ is a \emph{co-tile} of $F$ if every point of $\Z^d$ has a unique representation as $f+a$ with $f\in F$ and $a\in A$. We write
\[
 F\oplus A=\Z^d,
 \qquad\text{equivalently}\qquad
 \ind_F*\ind_A=1.
\]
Here $\ind_B$ denotes the indicator of a set $B$. A \emph{joint co-tile} of $F_1,\ldots,F_k$ is a single set $A$ satisfying $\ind_{F_i}*\ind_A=1$ for every $i$. A \emph{Boolean configuration} is a function from $\Z^d$ to $\{0,1\}$, and a finite tile is \emph{nontrivial} if it has at least two points.

A subgroup $L\leq\Z^d$ is a \emph{period subgroup} of a function $u$ if $u(x+h)=u(x)$ for every $x\in\Z^d$ and $h\in L$. Periodicity of a set means periodicity of its indicator. A set or function is \emph{fully periodic} if it has a finite-index period subgroup. Equivalently, it is invariant under $q\Z^d$ for some integer $q\geq1$; we call such a $q$ a \emph{scalar period}.

The periodic tiling problem asks whether a finite set admitting a co-tile must admit a fully periodic one. In one dimension, Newman~\cite{N77} proved the stronger assertion that every co-tile of a finite subset of $\Z$ is periodic. Bhattacharya~\cite{B20} established periodic existence in $\Z^2$. Greenfeld and Tao~\cite{GT21} subsequently proved that every planar co-tile is piecewise one-periodic and obtained quantitative bounds for periodic replacements. Their later counterexample~\cite{GT24} shows that periodic existence fails for general finite tiles in sufficiently high dimension. Thus an analogue of the planar existence theorem in arbitrary dimension requires additional hypotheses.

Meyerovitch, Sanadhya and Solomon~\cite{MSS24} developed a higher-dimensional framework based on a geometric independence condition for simultaneous tiling equations. We use their definition.

\Needspace{7\baselineskip}
\begin{definition}\label{def:independent}
An ordered tuple $(F_1,\ldots,F_k)$ of finite subsets of $\Z^d$ is \emph{independent} if $0\in F_i$ for every $i$, and every choice
\[
 v_i\in F_i\setminus\{0\}\qquad(1\leq i\leq k)
\]
is linearly independent over $\Q$.
\end{definition}

For every $t\in\Z^d$, the equations $F\oplus A=\Z^d$ and $(F-t)\oplus A=\Z^d$ are equivalent. Thus each tile can be normalized to contain $0$ without changing its co-tiles. Independence refers to the specified normalized sets; it is not asserted to be independent of the chosen normalizing points. If one tile in a family is a singleton, any joint co-tile must be $\Z^d$, and the other tiling equations then force every tile in the family to be a singleton.

Meyerovitch, Sanadhya and Solomon proved that every joint co-tile of $d$ independent tiles in $\Z^d$ is fully periodic~\cite[Theorem 1.2]{MSS24}. For $d-1$ independent tiles, they obtained periodic existence under property $(\star)$~\cite[Corollary 1.7]{MSS24}. Immediately after that corollary they asked whether $(\star)$ is necessary for the existence of a periodic joint co-tile~\cite[p.~5]{MSS24}. The question is already present in the first version of their paper.\footnote{The first version, \href{https://arxiv.org/abs/2301.11255v1}{arXiv:2301.11255v1}, was submitted on 26 January 2023; there the question follows Corollary 1.6. We use the theorem numbering of the published version throughout.} Our main theorem strengthens their existence result by retaining independence and dropping $(\star)$.

\begin{theorem}\label{thm:main}
Let $d\geq2$, and let $(F_1,\ldots,F_{d-1})$ be an independent tuple of nonempty finite subsets of $\Z^d$. If the tiles have a joint co-tile, then they have a fully periodic joint co-tile.
\end{theorem}

For $d=2$, independence of a one-tile tuple imposes no restriction beyond normalization, and Theorem~\ref{thm:main} recovers Bhattacharya's periodic existence theorem. For $d=3$, it gives periodic existence for two independent tiles. Thus $d-1$ independent tiles suffice for periodic existence, whereas $d$ independent tiles force every joint co-tile to be fully periodic. Example~\ref{ex:nonperiodic} shows that the distinction between these conclusions is necessary. The theorem concerns joint tilings and does not settle the periodic tiling problem for an arbitrary single tile of $\Z^3$.

For comparison, property $(\star)$ of~\cite[Definition 1.4]{MSS24} requires an independent $(d-1)$-tuple to satisfy
\begin{equation}\label{eq:star}
 \spanq(v_1,\ldots,v_{d-1})=\spanq(w_1,\ldots,w_{d-1})
 \quad\Longrightarrow\quad
 v_i=w_i\quad(1\leq i\leq d-2)
\end{equation}
for all nonzero choices $v_i,w_i\in F_i$. The condition depends on the order of the tuple. Under $(\star)$, Meyerovitch, Sanadhya and Solomon prove that every joint co-tile is a finite disjoint union of sets with rank-$(d-1)$ period subgroups~\cite[Theorem 1.6]{MSS24}, from which their Corollary 1.7 follows. Our theorem removes $(\star)$ from the existence conclusion, rather than asserting this decomposition for every original joint co-tile. Examples~\ref{ex:coordinate} and~\ref{ex:fullrank} separately show that $(\star)$ is not a necessary property of a periodically tiling independent tuple: it fails in every ordering of those tuples, although they have lattice joint co-tiles.

The main theorem follows from a periodic replacement principle that does not assume independence.

\begin{definition}\label{def:periodic}
For $0\leq j\leq d$, a function on $\Z^d$ is \emph{$j$-periodic} if it has a period subgroup of rank $j$. This does not require its full period group to have rank exactly $j$. In particular, \emph{codimension-one periodic} means invariant under some subgroup of rank $d-1$; fully periodic functions, including constants, also satisfy this condition. A set is \emph{$j$-periodic} if its indicator is $j$-periodic. It is \emph{piecewise $(d-1)$-periodic} if it is a finite disjoint union of $(d-1)$-periodic sets, whose period subgroups may differ from one piece to another.
\end{definition}

\begin{theorem}[Periodic replacement from a bounded decomposition]\label{thm:periodize}
Let $d\geq2$, and let $F_1,\ldots,F_k\subset\Z^d$ be nonempty finite sets with a joint co-tile $A$. Suppose that
\begin{equation}\label{eq:decomposition}
 \ind_A=c-\sum_{i=1}^r\varphi_i,
\end{equation}
where $c$ is a fully periodic real function and each $\varphi_i$ is a bounded real function invariant under a subgroup of rank $d-1$. Then the family has a piecewise $(d-1)$-periodic joint co-tile, and consequently a fully periodic joint co-tile.
\end{theorem}

Theorem 1.1 of~\cite{MSS24} supplies a decomposition of the form~\eqref{eq:decomposition} for joint co-tiles of $d-1$ independent tiles. Its summands are bounded real functions, and need not be indicators of disjoint sets. Theorem~\ref{thm:periodize} converts this real decomposition into a piecewise-periodic replacement. The final passage to a fully periodic joint co-tile is \cite[Theorem 1.3]{MSS24}.

The normal-form argument underlying Theorem~\ref{thm:periodize} applies to any Boolean configuration $a$ with the decomposition~\eqref{eq:decomposition}, without using a tiling equation. Theorem~\ref{thm:normal} produces a configuration $b$ in the translation orbit closure of $a$ and an integer $q\geq1$. For each residue class $x+q\Z^d$, consider the function $y\mapsto b(x+qy)$ on $\Z^d$. This function is either codimension-one periodic or has the form
\begin{equation}\label{eq:introE}
 2-\sum_{j=1}^3\{\alpha_j\ell_j(y)+\beta_j\},
\end{equation}
where $\{\cdot\}$ denotes fractional part, the $\ell_j:\Z^d\to\Z$ are pairwise nonparallel surjective integer linear forms, and $1,\alpha_j,\beta_j$ are linearly independent over $\Q$ for each $j$. The rational span of the three forms has dimension two. We call a Boolean function with this representation and these parameter conditions \emph{exceptional}. The proof extends the planar Boolean normalization method of Greenfeld and Tao~\cite[Section 8, especially Proposition 8.7]{GT25} to codimension-one period subgroups.

The tiling equations are used only after this normal form has been obtained. Every exceptional function has density one half. Neither of its level sets contains a nonempty codimension-one-periodic subset, and no translate of the function equals its complement. Fix a tile and an output residue class modulo $q$, and express its tiling equation as a sum of translates of the functions $y\mapsto b(x+qy)$. If one summand is exceptional, nonnegativity first forces every codimension-one-periodic summand to vanish identically. Averaging then shows that exactly two exceptional summands remain, counted according to the original tile points rather than as distinct functions.

We construct a graph with one vertex for each residue class on which the restriction is exceptional. Each tiling equation with exceptional summands gives an edge joining the residue classes of its two exceptional summands. The equation identifies a translate of one endpoint function with the complement of the other. An odd closed walk would therefore identify a translate of an exceptional function with its own complement, which is impossible. The graph is bipartite, and its vertices can be colored by $0$ and $1$ so that adjacent vertices receive different colors. Replacing each exceptional restriction by the constant equal to its vertex color, while leaving every codimension-one-periodic restriction unchanged, preserves all the tiling equations. The resulting joint co-tile is piecewise $(d-1)$-periodic, as required in Theorem~\ref{thm:periodize}.

Herva and Kari~\cite[Corollary 19]{HK26} proved that joint co-tile indicators of independent tiles admit decompositions into integer-valued periodic functions. Their theorem does not guarantee that the integer-valued summands are bounded, and therefore does not supply the bounded decomposition used here. Moreover, Example~\ref{ex:carry} shows that even a bounded real codimension-one decomposition of a Boolean configuration need not yield a disjoint periodic partition.

Several consequences are immediate once the simultaneous replacement principle is available.

\begin{corollary}[An independent subfamily suffices]\label{cor:subfamily}
Let $d\geq2$, and let $F_1,\ldots,F_k\subset\Z^d$ be nonempty finite sets. Suppose some $d-1$ of these sets, normalized to contain $0$, form an independent tuple. If the entire family has a joint co-tile, then it has a fully periodic joint co-tile.
\end{corollary}

\begin{corollary}[Independent companions]\label{cor:companions}
Let $d\geq2$ and $\{0\}\subsetneqq F\subset\Z^d$ be finite. Then $F$ admits a fully periodic co-tile if and only if there are finite sets $G_1,\ldots,G_{d-2}$ containing $0$ and a set $A\subset\Z^d$ such that $(G_1,\ldots,G_{d-2},F)$ is independent and
\[
 F\oplus A=G_1\oplus A=\cdots=G_{d-2}\oplus A=\Z^d.
\]
For $d=2$ the list of companions is empty.
\end{corollary}

The forward implication is contained in \cite[Theorem 1.9]{MSS24}; we recall a short proof in Section~\ref{sec:consequences}. The reverse implication removes $(\star)$ from the characterization in their Corollary 1.10. In particular, a nontrivial finite tile of $\Z^3$ has a periodic co-tile exactly when it has an independent companion with which it shares some co-tile.

\begin{corollary}[Uniform decidability]\label{cor:decision}
There is an algorithm that, given $d\geq2$ and a finite family of nonempty finite subsets of $\Z^d$ containing an independent $(d-1)$-subfamily, decides whether the family has a joint co-tile. On a positive instance it produces a scalar period and a fundamental pattern for a periodic joint co-tile. On a negative instance it produces a finite unsatisfiable collection of the tiling equations. The dimension may be part of the input.
\end{corollary}

The algorithm uses the standard parallel search for a periodic solution and a finite obstruction; compare \cite[Section 8.2]{MSS24}. Corollary~\ref{cor:subfamily} guarantees termination on positive instances; compactness guarantees termination on negative ones. Neither a decomposition nor an infinite co-tile is needed as input. We give no complexity bound.

For the rest of this paper, Section~\ref{sec:prelim} records the two inputs from~\cite{MSS24} and the averaging conventions,  Section~\ref{sec:normal} proves the orbit-closure normal form, and Section~\ref{sec:graph} proves periodic replacement and the main theorem. The consequences and examples occupy Sections~\ref{sec:consequences} and~\ref{sec:examples}. 

\section{Preliminaries and the joint-tiling inputs}\label{sec:prelim}
We use the translation convention $T^h u(x)=u(x+h)$ and write $\Delta_h=T^h-I$. For a finitely supported function $f$,
\[
 (f*u)(x)=\sum_{v\in\Z^d}f(v)u(x-v).
\]
For a function $u$, its support is $\supp u=\{x:u(x)\ne0\}$; for a Boolean configuration this is the level set $\{u=1\}$. The translation orbit closure of a Boolean configuration is taken in the product topology on $\{0,1\}^{\Z^d}$. Since $\Z^d$ is countable, this is a compact metrizable space. For a subgroup $\Gamma$ we write
\[
 \overline{\orb_\Gamma(u)}
 =\overline{\{T^h u:h\in\Gamma\}}.
\]
For a tuple of functions, this notation refers to simultaneous translation of all its coordinates by the same vector. The space of joint co-tile indicators of a fixed finite family is closed and translation invariant: each equation depends on finitely many coordinates. Consequently every point in the orbit closure of a joint co-tile indicator is again a joint co-tile indicator.

For $q\geq1$ and $x\in\Z^d$, the \emph{residue restriction} of $u$ to $x+q\Z^d$ means the function
\[
 u_{x,q}(y)=u(x+qy),\qquad y\in\Z^d.
\]
Thus a residue restriction is a function on a new copy of $\Z^d$, not a function on the finite quotient $(\Z/q\Z)^d$. A period subgroup $L$ of $u_{x,q}$ gives the period subgroup $qL$ for the restriction of $u$ on the original coset. Choosing another representative $x+qz$ of the same coset replaces $u_{x,q}(y)$ by $u_{x,q}(y+z)$.

We use the following two theorems with their published hypotheses.

\begin{theorem}[Meyerovitch--Sanadhya--Solomon]\label{thm:MSSdecomp}
Let $1\leq k\leq d$, and let $(F_1,\ldots,F_k)$ be an independent tuple of nontrivial finite subsets of $\Z^d$. Every joint co-tile indicator is an integer constant plus a finite sum of bounded real functions, each invariant under a subgroup of rank $k$.
\end{theorem}

This is the portion of \cite[Theorem 1.1]{MSS24} that we need. The refined decomposition in their Theorem 3.1 specifies the period subgroups. When $k=d-1$, changing the signs of the summands gives~\eqref{eq:decomposition}. When $k=d$, the finite sum is fully periodic.

\begin{theorem}[Meyerovitch--Sanadhya--Solomon]\label{thm:MSSreplacement}
Let $F_1,\ldots,F_k\subset\Z^d$ be nonempty finite sets. If they have a piecewise $(d-1)$-periodic joint co-tile, then they have a fully periodic joint co-tile.
\end{theorem}

This is \cite[Theorem 1.3]{MSS24}. It requires neither independence nor property $(\star)$.

\subsection{Period hyperplanes and invariant averages}
An integer linear form $\ell:\Z^d\to\Z$ is \emph{primitive} if it is surjective. Equivalently, the greatest common divisor of its integer coefficients is one. We identify such forms with their rational linear extensions to $\Q^d$. Two nonzero forms are \emph{parallel} if one is a rational multiple of the other, or equivalently if their kernels in $\Z^d$ coincide. The rational rank of a collection of forms is the dimension of their span in the dual space $(\Q^d)^*$.

If $L\leq\Z^d$ has rank $d-1$, its saturation
\[
 H=(\spanq L)\cap\Z^d
\]
is the kernel of a surjective integer linear form $\ell:\Z^d\to\Z$. The index $[H:L]$ is finite. Given finitely many such subgroups, a common scalar multiple of each saturation lies in the respective original subgroup. We combine components having the same rational period hyperplane. Thus, in~\eqref{eq:decomposition}, we can assume
\begin{equation}\label{eq:hyperplanes}
 H_i=\ker\ell_i,\qquad
 \ell_i(\Z^d)=\Z,\qquad
 H_i\ne H_j\ (i\ne j),
\end{equation}
with $\varphi_i$ invariant under $q_0H_i$ for a common $q_0\geq1$. We also choose $q_0$ so that $c$ is $q_0\Z^d$-periodic. Distinctness implies that $H_i+H_j$ has finite index in $\Z^d$. It also implies that $H_j$ contains a vector on which $\ell_i$ is nonzero. These two facts supply the transversality needed below.

For a subgroup $L$, let $\pi_L$ be an invariant averaging operator on bounded functions: choose a translation-invariant mean $m_L$ on $\ell^\infty(L)$ and set
\[
 (\pi_Lu)(x)=m_L\bigl(h\mapsto u(x+h)\bigr).
\]
Such a mean can be obtained from a generalized limit of averages over F\o lner boxes in $L$. The operator is linear and positive, preserves constants, commutes with all translations, fixes $L$-periodic functions, and has $L$-periodic range. If $u$ is also $L'$-periodic, then $\pi_Lu$ is $(L+L')$-periodic. No countable additivity of this averaging operator is used; the Borel probability measure used later is a separate construction.

\begin{lemma}[Range normalization]\label{lem:range}
Suppose a Boolean configuration $a$ has a decomposition~\eqref{eq:decomposition} with the periods~\eqref{eq:hyperplanes}. Without changing $a$ or the hyperplanes, one can arrange $0\leq\varphi_i\leq1$ for every $i$, after changing $c$ by a fully periodic real function and increasing $q_0$.
\end{lemma}
\begin{proof}
Averaging the decomposition along $q_0H_i$ gives
\[
 \varphi_i=-\pi_{q_0H_i}a+c-
 \sum_{j\ne i}\pi_{q_0H_i}\varphi_j.
\]
Every cross term is fully periodic, since it is invariant under $q_0H_i+q_0H_j$. Hence $\varphi_i$ differs from
\[
 \psi_i=1-\pi_{q_0H_i}a\in[0,1]
\]
by a fully periodic function. Set $\rho_i=\varphi_i-\psi_i$. Then
$a=(c-\sum_i\rho_i)-\sum_i\psi_i$.
All corrections are computed from the original representation, so these replacements can be made simultaneously. A common increase of $q_0$ accommodates their periods.
\end{proof}

\subsection{Polynomial phases and box averages}
We write $\{t\}\in[0,1)$ for fractional part. A real polynomial on $\Z^d$ with an irrational nonconstant coefficient is uniformly distributed modulo one along growing boxes. More generally, a finite vector of polynomial phases is uniformly distributed on the corresponding torus if every nonzero integer combination has an irrational nonconstant coefficient. We use this standard form of Weyl's theorem, as in~\cite[Section 8]{GT25}. In particular, fractional-part functions have first moment $1/2$ and second moment $1/3$, and two jointly equidistributed fractional-part functions have mixed moment $1/4$.

To specify the box averages, fix a basis $v_1,\ldots,v_j$ of a subgroup $L$ of rank $j\geq1$. For a bounded function $u$, set
\[
 \langle u\rangle_{x+L}
 =\lim_{N\to\infty}\frac{1}{(2N+1)^j}
   \sum_{n_1,\ldots,n_j=-N}^{N}
   u\left(x+\sum_{a=1}^{j}n_av_a\right)
\]
whenever the limit exists. We omit the subscript for averages on $\Z^d$. The \emph{density} of a Boolean function or of its support means this box average. These limits are distinct from the invariant means used in the preceding subsection; in each subsequent use, the relevant limits are shown to exist. The polynomial equidistribution conclusions hold after restriction to any coset of a finite-index subgroup. We will also use the elementary fact that an integer-valued real polynomial in one variable has rational coefficients, which follows by Newton interpolation at the integers.

\Needspace{22\baselineskip}
\section{A Boolean normal form in translation orbit closures}\label{sec:normal}
The normal form in this section is independent of the tiling equations. Its two-dimensional prototype is \cite[Proposition 8.7]{GT25}. The proof below extends the endpoint selection, polynomial-phase, and moment arguments to distinct codimension-one period subgroups.

\begin{theorem}[Orbit-closure normal form]\label{thm:normal}
Let $d\geq2$, and suppose $a:\Z^d\to\{0,1\}$ has a decomposition
\[
 a=c-\sum_{i=1}^r\varphi_i
\]
with $c$ fully periodic and real-valued, and with each $\varphi_i$ bounded, real-valued, and codimension-one periodic. For every prescribed integer $p\geq1$, there are
\[
 b\in\overline{\orb_{p\Z^d}(a)}
 \quad\text{and}\quad q\geq1,\qquad p\mid q,
\]
such that, for each coset $C=x+q\Z^d$, the function $b_C:\Z^d\to\{0,1\}$ defined by $b_C(y)=b(x+qy)$ satisfies one of the following alternatives:
\begin{enumerate}[label=\textup{(\roman*)}]
\item \emph{Codimension-one periodic restriction.} There is a subgroup $L_C\leq\Z^d$ of rank $d-1$ such that
\[
 b_C(y+h)=b_C(y)\qquad(y\in\Z^d,\ h\in L_C).
\]
The subgroup may depend on $C$, and the constant functions $0$ and $1$ are included in this alternative.
\item \emph{Exceptional affine restriction.} There are three pairwise nonparallel surjective integer linear forms $\ell_1,\ell_2,\ell_3:\Z^d\to\Z$ such that
\begin{equation}\label{eq:normalE}
 b_C(y)=2-\sum_{j=1}^3\{\alpha_j\ell_j(y)+\beta_j\},
\end{equation}
where $1,\alpha_j,\beta_j$ are linearly independent over $\Q$ for each $j$, and the rational span of $\ell_1,\ell_2,\ell_3$ has dimension two.
\end{enumerate}
The forms $\ell_j$ in \textup{(ii)} may be chosen from the primitive integer linear forms defining the distinct saturated period hyperplanes in the original decomposition. The choice of these three forms and of the parameters $\alpha_j,\beta_j$ may depend on $C$. The conclusions persist under replacement of $q$ by any multiple.
\end{theorem}

In particular, the normal form can be obtained inside any closed translation-invariant space containing $a$. For a closed space invariant only under a prescribed finite-index subgroup, choose $p\Z^d$ inside that subgroup. This also permits preservation of finite convolution equations with a fixed periodic right-hand side.

The proof uses three lemmas: polynomial phases on lattice cosets, a stationary selection with controlled endpoint values and affine phases, and pairwise moment identities. Throughout these lemmas, components with the same period hyperplane have been combined as in~\eqref{eq:hyperplanes}, and their ranges have been normalized by Lemma~\ref{lem:range}.

\begin{lemma}[Polynomial phases on lattice cosets]\label{lem:polynomial}
Suppose
\[
 a=c-\sum_{i=1}^r\varphi_i,\qquad
 a\in\{0,1\}^{\Z^d},\quad 0\leq\varphi_i\leq1,
\]
where $c$ is $q_0\Z^d$-periodic and $\varphi_i$ is $q_0H_i$-periodic for the distinct hyperplanes~\eqref{eq:hyperplanes}. There is a multiple $q$ of $q_0$ such that, on each coset $C=x+q\Z^d$, each component either is constant modulo one or satisfies
\begin{equation}\label{eq:polynomialphase}
 \varphi_i(x+qy)\equiv P_{i,C}(\ell_i(y))\pmod1,
\end{equation}
where $P_{i,C}$ is a real polynomial of positive degree with irrational leading coefficient. This dichotomy persists on further scalar lattice cosets. A component is \emph{constant modulo one} on $C$ when its fractional part is independent of $y$; the component itself need not be constant. In that alternative, the component is either constant as a real-valued function or takes only the values $0$ and $1$ on $C$.
\end{lemma}
\begin{proof}
The case $r=0$ is empty. For $r=1$, integrality of $a$ makes $\varphi_1$ constant modulo one on every $q_0\Z^d$ coset. Suppose $r\geq2$, and fix $i$. For every $j\ne i$, choose $v_j\in H_j$ such that $r_j=\ell_i(v_j)\ne0$. Reduce the decomposition modulo one and apply the commuting differences
\[
 \prod_{j\ne i}\Delta_{q_0v_j}.
\]
They annihilate the constant term and every component other than $\varphi_i$. On a fixed coset $x+q_0\Z^d$, invariance under $q_0H_i$ gives a well-defined sequence
\[
 \varphi_i(x+q_0y)\bmod1=u_i(\ell_i(y)),
 \qquad u_i:\Z\to\R/\Z.
\]
Surjectivity of $\ell_i$ then gives
\begin{equation}\label{eq:mixed1ddiff}
 \prod_{j\ne i}\Delta_{r_j}u_i=0.
\end{equation}
Let $R$ be a positive common multiple of all $|r_j|$. In the Laurent polynomial ring of the one-dimensional shift, each $\Delta_{r_j}$ divides $\Delta_R$, including when $r_j<0$. Multiplying these divisibility identities in~\eqref{eq:mixed1ddiff} yields $\Delta_R^{r-1}u_i=0$; no division in $\R/\Z$ is involved. Newton's finite-difference formula now gives, for each $a\in\{0,\ldots,R-1\}$,
\begin{equation}\label{eq:Newtoncircle}
 u_i(a+Rn)=\sum_{h=0}^{r-2}\binom{n}{h}(\Delta_R^h u_i)(a)
 \quad\text{in }\R/\Z,\qquad n\in\Z.
\end{equation}
The formula holds also for negative $n$, since the binomial polynomials satisfy the same difference recurrence in both directions. Arbitrary real lifts of the finitely many coefficients in~\eqref{eq:Newtoncircle} give a real polynomial representative of degree at most $r-2$ on each progression.

To pass from these progressions to scalar lattice cosets, fix $z\in\Z^d$ and write $\ell_i(z)=a+Rt$ with $0\leq a<R$. If $P_a$ represents $u_i(a+Rn)$, then
\[
 \varphi_i(x+q_0z+q_0Ry)
 \equiv P_a\bigl(t+\ell_i(y)\bigr)\pmod1.
\]
Thus each progression representation gives a polynomial in the same primitive form $\ell_i$ on the corresponding scalar lattice coset.

There are finitely many indices and initial cosets, so a common scalar refinement gives polynomial representatives for every component on every coset. A real polynomial with rational nonconstant coefficients is constant modulo one on sufficiently fine arithmetic progressions: a common denominator of those coefficients provides such a modulus. For a polynomial having an irrational nonconstant coefficient, remove its rational terms above the largest irrational degree in this way. On each resulting progression the remaining representative has positive degree and irrational leading coefficient. Taking a common refinement for the finitely many polynomials proves~\eqref{eq:polynomialphase}.

An integer affine restriction with nonzero step preserves irrationality of the leading coefficient. This proves persistence under further refinement. Finally, a $[0,1]$-valued function with constant fractional part is constant unless that fractional part is zero, in which case its values lie in $\{0,1\}$.
\end{proof}

\begin{lemma}[Stationary selection]\label{lem:stationary}
Under the hypotheses of Lemma~\ref{lem:polynomial}, there is a tuple
\[
 (b,\psi_1,\ldots,\psi_r)
 \in\overline{\orb_{q_0\Z^d}(a,\varphi_1,\ldots,\varphi_r)}
\]
with the following properties. There is a multiple $Q$ of $q_0$ such that:
\begin{enumerate}[label=\textup{(\alph*)}]
\item For every $i$ and every $x,y\in\Z^d$ with $y-x\in Q\Z^d$, the conditions $\psi_i(x)=0$ and $\psi_i(y)=1$ imply $b(x)=1$ and $b(y)=0$.
\item For every $i$, every multiple $q$ of $q_0$, every $x\in\Z^d$, and every pair $\alpha,\beta\in\R$, if $\alpha$ is irrational and
\[
 \psi_i(x+qy)\equiv\alpha\ell_i(y)+\beta\pmod1
 \qquad\text{for all }y\in\Z^d,
\]
then $1,\alpha,\beta$ are linearly independent over $\Q$.
\end{enumerate}
The selected tuple satisfies $b=c-\sum_i\psi_i$, with each $\psi_i$ taking values in $[0,1]$ and invariant under $q_0H_i$.
\end{lemma}
\begin{proof}
Let $X$ be the indicated orbit closure in the compact metrizable space
\[
 \{0,1\}^{\Z^d}\times\prod_{i=1}^r[0,1]^{\Z^d}.
\]
The decomposition $b=c-\sum_i\psi_i$ and the periods of the components are closed conditions preserved by $q_0\Z^d$ translations; $c$ is fixed by these translations. Hence these conditions hold throughout $X$. Averaging point masses over F\o lner boxes in $q_0\Z^d$ and taking a weak limit gives an invariant Borel probability measure $\mu$ on $X$.

Choose $Q$ divisible by $q_0$ such that
\begin{equation}\label{eq:Qchoice}
 Q\Z^d\subset q_0H_i+q_0H_j\qquad(i\ne j).
\end{equation}
If there is at most one component, take $Q=q_0$. Fix $i,x,y$ with $y-x\in Q\Z^d$, and let
\[
 \mathcal D_{i,x,y}
 =\{(b,\psi_1,\ldots,\psi_r)\in X:
       \psi_i(x)=0,\ \psi_i(y)=1\}.
\]
The event $\mathcal D_{i,x,y}$ is invariant under $q_0H_i$. On this event the decomposition gives
\begin{equation}\label{eq:endpointdifference}
 b(x)-b(y)=1+
 \sum_{j\ne i}\bigl(\psi_j(y)-\psi_j(x)\bigr).
\end{equation}
For each $j\ne i$, write $y-x=u_j+w_j$ with $u_j\in q_0H_i$ and $w_j\in q_0H_j$, using~\eqref{eq:Qchoice}. The $j$th summand in~\eqref{eq:endpointdifference} equals $\psi_j(x+u_j)-\psi_j(x)$. Its integral over $\mathcal D_{i,x,y}$ vanishes by invariance of both $\mu$ and $\mathcal D_{i,x,y}$ under $u_j$. Consequently
\[
 \int_{\mathcal D_{i,x,y}}\bigl(b(x)-b(y)\bigr)\,d\mu=\mu(\mathcal D_{i,x,y}).
\]
Since $b(x)-b(y)\leq1$, equality forces $b(x)=1$, $b(y)=0$ almost surely on $\mathcal D_{i,x,y}$. Intersecting the resulting full-measure sets over the countably many $i,x,y$ proves (a) almost surely.

For (b), choose $e_i\in\Z^d$ with $\ell_i(e_i)=1$. Fix $i$, a multiple $q$ of $q_0$, and $x\in\Z^d$. Representatives of the affine parameters in $[0,1)$, whenever the affine identity holds, are recovered by the Borel formulas
\begin{equation}\label{eq:measurablephases}
 \beta=\{\psi_i(x)\},\qquad
 \alpha=\{\psi_i(x+qe_i)-\psi_i(x)\}.
\end{equation}
The affine identity is a countable collection of coordinate equalities modulo one, so the event that it holds with irrational $\alpha$ is Borel. For fixed integers $A,B,D$ with $B\ne0$, consider the subevent
\[
 A\alpha+B\beta+D=0.
\]
Its translates under $nqe_i$, $n\in\Z$, are pairwise disjoint. Indeed, along such a translation orbit the slope $\alpha$ is fixed, whereas the intercept becomes $\{\beta+n\alpha\}$. If the same relation holds at distinct times $n,m$, write $\{\beta+j\alpha\}=\beta+j\alpha-k_j$ with $k_j\in\Z$ for $j=n,m$. Subtracting the two relations gives
\[
 B(n-m)\alpha=B(k_n-k_m),
\]
contrary to irrationality. Invariance and finiteness of $\mu$ imply that this event has measure zero. Nontrivial relations with $B=0$ are excluded by irrationality of $\alpha$.

There are only countably many indices, scalar moduli, lattice points representing cosets, and integer relations. We can therefore select one tuple for which (a) and (b) hold simultaneously for all of them. Changing representatives of $\alpha$ or $\beta$ by integers does not affect the claimed rational independence. Replacing $x$ by another representative $x+qz$ changes the intercept by $\alpha\ell_i(z)$ modulo one, which also preserves this independence.
\end{proof}

\begin{lemma}[Pairwise moments]\label{lem:moments}
Suppose $U_i:\Z^d\to[0,1]$ satisfies
\[
 U_i(y)\equiv P_i(\ell_i(y))\pmod1,
\]
where the $\ell_i$ have pairwise distinct kernels and every $P_i$ has positive degree and irrational leading coefficient. On $\Z^d$ and on every finite-index affine coset,
\begin{equation}\label{eq:moments}
 \langle U_i\rangle=\frac12,\qquad
 \langle U_i^2\rangle=\frac13,\qquad
 \langle U_iU_j\rangle=\frac14\quad(i\ne j).
\end{equation}
In particular, the variance of the sum of $R$ such functions is $R/12$.
\end{lemma}
\begin{proof}
Each polynomial phase is equidistributed modulo one. To check joint equidistribution of a pair, take integers $A,B$, not both zero. If $A\ne0$, restrict
\[
 A P_i(\ell_i(y))+B P_j(\ell_j(y))
\]
to an integer line in $\ker\ell_j$ on which $\ell_i$ is nonzero. The restricted polynomial has positive degree and irrational leading coefficient. The original multivariate polynomial must therefore have an irrational nonconstant coefficient. The case $A=0$ follows directly from the second phase. The same argument works on any finite-index affine coset, using a direction in its period lattice intersected with $\ker\ell_j$.

Weyl's theorem now gives joint uniform distribution on $(\R/\Z)^2$. The fractional-part functions and their products are Riemann integrable on the torus. Integral phase values form a set of density zero, so choosing the value $1$ instead of $0$ at such points does not affect any of the moments. This proves~\eqref{eq:moments}. Expanding the second moment of the sum gives the stated variance; only pairwise joint equidistribution is used.
\end{proof}

\begin{proof}[Proof of Theorem~\ref{thm:normal}]
If there are no components, choose a scalar period $q$ of $a=c$ divisible by $p$ and take $b=a$. Otherwise combine period hyperplanes and apply Lemma~\ref{lem:range}. Choose the initial modulus $q_0$ divisible by $p$ and by all required scalar periods. Apply Lemma~\ref{lem:stationary} to obtain a tuple $(b,\psi_1,\ldots,\psi_r)$. Its Boolean coordinate belongs to $\overline{\orb_{p\Z^d}(a)}$, and
\[
 b=c-\sum_i\psi_i.
\]

Apply Lemma~\ref{lem:polynomial} to the selected tuple, and take its resulting modulus $q$ divisible by $Q$. This order is deliberate: the polynomial refinement may depend on the selected tuple, while the conclusions of Lemma~\ref{lem:stationary} were arranged simultaneously for every possible scalar modulus.

Fix a coset $C=x+q\Z^d$. A component with constant fractional part is either constant as a real-valued function or Boolean there. Suppose $\psi_i$ is Boolean and nonconstant on $C$, and choose $z_0,z_1\in C$ with $\psi_i(z_0)=0$ and $\psi_i(z_1)=1$. If $z\in C$ satisfies $\psi_i(z)=0$, apply Lemma~\ref{lem:stationary}(a) to $z,z_1$ to obtain $b(z)=1$. If $\psi_i(z)=1$, apply the same lemma to $z_0,z$ to obtain $b(z)=0$. These applications are valid because $q$ is a multiple of $Q$. Thus
\begin{equation}\label{eq:Booleanidentification}
 b=1-\psi_i\qquad\text{on }C.
\end{equation}
Since $\psi_i$ is invariant under $qH_i$, equation~\eqref{eq:Booleanidentification} makes $b_C$ invariant under $H_i$. This is alternative (i), so no further classification of the other components on this coset is needed.

We may therefore assume that every component with constant fractional part is constant as a real-valued function on $C$. Subtract these constant component values from the constant value of $c$ on $C$, and denote the result by $K$. If there are no remaining components, $b_C$ is constant and again satisfies (i). Otherwise, we have
\begin{equation}\label{eq:Ksum}
 b_C=K-\sum_{j=1}^{R}U_j,\qquad R\geq1,
\end{equation}
where $K$ is constant and every $U_j$ has an irrational polynomial phase, with the defining linear forms having pairwise distinct kernels. Let $\theta=\langle b_C\rangle=K-R/2$ and $S=\sum_{j=1}^R U_j$. Lemma~\ref{lem:moments} gives
\[
 \langle S\rangle=\frac R2,
 \qquad
 \langle S^2\rangle=\frac R3+\frac{R(R-1)}4.
\]
Consequently $\langle S^2\rangle-\langle S\rangle^2=R/12$. Since $b_C=K-S$ and $b_C^2=b_C$, it follows that
\begin{equation}\label{eq:variancebound}
 \theta(1-\theta)=\frac{R}{12}.
\end{equation}
As $0\leq\theta\leq1$, this forces $R\leq3$. Only the individual second moments and the pairwise mixed moments have been used; joint equidistribution of all $R$ phases is not required.

We rule out $R=1$ and $R=2$. With one component, reducing~\eqref{eq:Ksum} modulo one forces its phase to be constant modulo one. With two components, choose a vector $v\in\ker\ell_2$ with $s=\ell_1(v)\ne0$. Taking a difference of~\eqref{eq:Ksum} modulo one gives
\[
 \Delta_sP_1(n)\in\Z\qquad(n\in\Z).
\]
The polynomial on the left has irrational leading coefficient, including when it is a constant difference of an affine polynomial. It cannot be integer-valued on all of $\Z$. Thus $R=3$.

For each of these three phases, choose directions in the other two kernels on which its own defining linear form is nonzero. Taking the two corresponding differences modulo one gives
\begin{equation}\label{eq:secondphasedifference}
 \Delta_s\Delta_tP_i(n)\in\Z\qquad(n\in\Z),\qquad st\ne0.
\end{equation}
If $\deg P_i\geq2$, the left side again has irrational leading coefficient and cannot be integer-valued. Hence every phase is affine with irrational slope. Lemma~\ref{lem:stationary}(b) gives rational independence of $1,\alpha_i,\beta_i$. No affine phase can take an integer value at any lattice point, so the $[0,1]$-valued components agree pointwise with
\[
 U_i(y)=\{\alpha_i\ell_i(y)+\beta_i\}.
\]
Equation~\eqref{eq:variancebound} gives $\theta=1/2$, and then $K=2$ in~\eqref{eq:Ksum}.

It remains to check the rational rank of the three forms. Reducing~\eqref{eq:normalE} modulo one and subtracting its value at zero shows that
\[
 \sum_{i=1}^3\alpha_i\ell_i(y)\in\Z
 \qquad(y\in\Z^d).
\]
If the three integer linear forms were rationally independent, three independent columns of their coefficient matrix would give an invertible rational system for the $\alpha_i$. Since the displayed linear form has integer coefficients, all three $\alpha_i$ would be rational. This is impossible. The forms are pairwise nonparallel, so their rational rank is exactly two.

The argument applies to each of the finitely many cosets for the modulus already chosen, so the alternatives hold simultaneously. To check persistence, refine $C=x+q\Z^d$ to a coset $x+qz+mq\Z^d$. The new restriction is $y\mapsto b_C(z+my)$. For an exceptional restriction, its $i$th affine phase becomes
\[
 m\alpha_i\ell_i(y)+\beta_i+\alpha_i\ell_i(z).
\]
Rational independence of $1,\alpha_i,\beta_i$ implies rational independence of $1,m\alpha_i,\beta_i+\alpha_i\ell_i(z)$. For a codimension-one-periodic restriction with period subgroup $L$, the new period subgroup is
\[
 \{h\in\Z^d:mh\in L\}=m^{-1}(L\cap m\Z^d),
\]
which still has rank $d-1$. This proves persistence and completes the proof.
\end{proof}

\begin{remark}\label{rem:ranktwo}
Every exceptional restriction~\eqref{eq:normalE} is invariant under the common kernel of its three defining linear forms, a subgroup of rank $d-2$. Hence the restriction factors through a quotient of $\Z^d$ of rank two. In particular, the term exceptional does not mean that the function has no nonzero periods. The proof does not require a choice of quotient coordinates or a compatibility condition between the rank-two spaces arising on different residue classes.
\end{remark}

\section{Half-density functions and periodic replacement}\label{sec:graph}
We first record the precise class of functions that occurs in the second alternative of Theorem~\ref{thm:normal}. A Boolean function $u:\Z^d\to\{0,1\}$ is called \emph{exceptional} if there are real numbers $\alpha_j,\beta_j$ and surjective integer linear forms $\ell_j:\Z^d\to\Z$, for $j=1,2,3$, such that
\begin{equation}\label{eq:E}
 u(y)=2-U_1(y)-U_2(y)-U_3(y),\qquad
 U_j(y)=\{\alpha_j\ell_j(y)+\beta_j\},
\end{equation}
and the following conditions hold: the forms are pairwise nonparallel, their rational span has dimension two, and $1,\alpha_j,\beta_j$ are linearly independent over $\Q$ for each $j$. The requirement that $u$ be Boolean is part of the definition; the parameter conditions alone are not asserted to imply it.

Rational independence implies that no $\alpha_j\ell_j(y)+\beta_j$ is an integer, so $0<U_j(y)<1$ at every lattice point. Each $U_j$ has box average $1/2$, and therefore $\langle u\rangle=1/2$. In particular, both level sets of $u$ are nonempty. Translates of exceptional functions are exceptional: translation by $t\in\Z^d$ replaces $\beta_j$ by $\beta_j+\alpha_j\ell_j(t)$ and preserves all the stated conditions. The next two lemmas provide the properties needed for simultaneous replacement.

\begin{lemma}[No periodic subset of a level set]\label{lem:Elevels}
Let $u$ be exceptional. Neither $\{u=0\}$ nor $\{u=1\}$ contains a nonempty codimension-one-periodic subset.
\end{lemma}
\begin{proof}
Let $L\leq\Z^d$ have rank at least $d-1$, and fix $z\in\Z^d$. At most one of the three forms $\ell_j$ vanishes on $L$: two distinct kernels intersect in rank $d-2$. If $\ell_j$ does not vanish on $L$, its image is a nonzero subgroup of $\Z$, and the irrational affine phase defining $U_j$ is uniformly distributed modulo one along boxes in $L$. Thus the average of $U_j$ on $z+L$ is $1/2$. If $\ell_j$ vanishes on $L$, the function is constant there with value $U_j(z)$. It follows that
\begin{equation}\label{eq:hyperplanemean}
 \langle u\rangle_{z+L}=
 \begin{cases}
 \frac12,&\text{if no $\ell_j$ vanishes on $L$},\\[2pt]
 1-U_j(z),&\text{if $\ell_j$ vanishes on $L$}.
 \end{cases}
\end{equation}
Both alternatives lie strictly between zero and one. Hence $u$ is not constant on $z+L$. A nonempty set invariant under $L$ contains a whole coset of $L$, which proves the assertion for both level sets.
\end{proof}

\begin{lemma}[No translation to the complement]\label{lem:noantiperiod}
An exceptional function has no translate equal to its complement: there is no $t\in\Z^d$ such that $u(y+t)=1-u(y)$ for all $y$.
\end{lemma}
\begin{proof}
Suppose such a $t$ exists. Average this identity over translates of $H_1=\ker\ell_1$. The other two fractional-part functions have mean $1/2$ on every such coset, so~\eqref{eq:hyperplanemean} gives
\[
 U_1(y+t)+U_1(y)=1\qquad(y\in\Z^d).
\]
Reduce modulo one and subtract the identity at $y$ from the identity at $y+e$, where $\ell_1(e)=1$. This gives $2\alpha_1\in\Z$, contradicting irrationality.
\end{proof}

\begin{proposition}[Simultaneous replacement of exceptional restrictions]\label{prop:graph}
Let $F_1,\ldots,F_k\subset\Z^d$ be nonempty finite sets, and let $b:\Z^d\to\{0,1\}$ satisfy $\ind_{F_i}*b=1$ for every $i$. Suppose that, for some $q\geq1$ and every $r\in\{0,\ldots,q-1\}^d$, the residue restriction
\[
 b_r(y)=b(r+qy),\qquad y\in\Z^d,
\]
is either invariant under a subgroup of rank $d-1$ or exceptional in the sense of~\eqref{eq:E} and its accompanying conditions. Then the family has a piecewise $(d-1)$-periodic joint co-tile.
\end{proposition}
\begin{proof}
Let $\mathcal R=\{0,\ldots,q-1\}^d$ be the chosen set of residue representatives, and define
\[
 \begin{aligned}
 \mathcal R_{\mathrm{per}}
   &=\{r\in\mathcal R:b_r\text{ is codimension-one periodic}\},\\
 \mathcal R_{\mathrm{exc}}
   &=\{r\in\mathcal R:b_r\text{ is exceptional}\}.
 \end{aligned}
\]
Residues with $b_r\equiv0$ or $b_r\equiv1$ belong to $\mathcal R_{\mathrm{per}}$. The two sets are disjoint: an exceptional function has nonempty support, whereas Lemma~\ref{lem:Elevels} prevents that support from being codimension-one periodic. By hypothesis they exhaust $\mathcal R$. We will leave $b_r$ unchanged for $r\in\mathcal R_{\mathrm{per}}$ and replace it by a constant for $r\in\mathcal R_{\mathrm{exc}}$.

Fix $i\in\{1,\ldots,k\}$ and $r\in\mathcal R$. For each tile point $f\in F_i$, there are unique vectors $s(r,f)\in\mathcal R$ and $h(r,f)\in\Z^d$ such that
\[
 r-f=s(r,f)+qh(r,f).
\]
Here $r$ represents the output residue class where the tiling equation is evaluated, whereas $s(r,f)$ represents the residue class sampled by its summand indexed by $f$. Substituting $x=r+qy$ into $\sum_{f\in F_i}b(x-f)=1$ gives
\begin{equation}\label{eq:occurrences}
 \sum_{f\in F_i}b_{s(r,f)}(y+h(r,f))=1
 \qquad(y\in\Z^d).
\end{equation}
Thus a summand in this equation is a translate of a residue restriction. The sum is indexed by the original set $F_i$, not by its image in $(\Z/q\Z)^d$: distinct tile points contribute separate summands even when they give the same residue restriction or the same translated function.

Let
\[
 J_{i,r}=\{f\in F_i:s(r,f)\in\mathcal R_{\mathrm{exc}}\}
\]
be the set of tile points indexing exceptional summands in~\eqref{eq:occurrences}. Suppose $J_{i,r}$ is nonempty, choose $f_0\in J_{i,r}$, and put
\[
 u(y)=b_{s(r,f_0)}(y+h(r,f_0)).
\]
This function is exceptional because it is a translate of an exceptional restriction. All summands in~\eqref{eq:occurrences} are Boolean and nonnegative. Hence the support of every summand other than the one indexed by $f_0$ is contained in $\{u=0\}$. If $f\notin J_{i,r}$, its summand has codimension-one-periodic support. Lemma~\ref{lem:Elevels} therefore forces that summand to vanish identically. The equation reduces to
\[
 \sum_{f\in J_{i,r}}b_{s(r,f)}(y+h(r,f))=1.
\]
Every summand in this reduced equation has box average $1/2$. Averaging gives $|J_{i,r}|/2=1$, and thus
\[
 |J_{i,r}|=2\qquad\text{whenever }J_{i,r}\ne\varnothing.
\]
This equality counts tile points, so it retains any multiplicities arising from repeated residues.

Construct a finite undirected multigraph $\mathcal G$ with vertex set $\mathcal R_{\mathrm{exc}}$. For each pair $(i,r)$ with $J_{i,r}\ne\varnothing$, write $J_{i,r}=\{f_1,f_2\}$ and insert one edge joining
\[
 s=s(r,f_1)\quad\text{and}\quad t=s(r,f_2).
\]
Different pairs $(i,r)$ may give parallel edges; at this stage a loop is also allowed. There are at most $kq^d$ edges. Write $h_1=h(r,f_1)$ and $h_2=h(r,f_2)$. The reduced tiling equation is
\[
 b_s(y+h_1)+b_t(y+h_2)=1.
\]
After replacing $y$ by $y-h_1$, it becomes
\begin{equation}\label{eq:edgeidentity}
 b_t(y+\delta)=1-b_s(y),\qquad
 \delta=h_2-h_1.
\end{equation}
Attach the shift $\delta$ to this edge when traversed from $s$ to $t$, and attach $-\delta$ when it is traversed in the opposite direction.

For a walk with vertices $s_0,s_1,\ldots,s_m$ and successive traversal shifts $\delta_1,\ldots,\delta_m$, repeated application of~\eqref{eq:edgeidentity} gives
\begin{equation}\label{eq:walkcomposition}
 b_{s_m}\left(y+\sum_{j=1}^m\delta_j\right)
 =\frac{1-(-1)^m}{2}+(-1)^m b_{s_0}(y).
\end{equation}
If the walk is closed and $m$ is odd, this identity says that a translate of the exceptional function $b_{s_0}$ equals $1-b_{s_0}$. Lemma~\ref{lem:noantiperiod} rules this out. Thus $\mathcal G$ has no odd closed walk, including no loop, and is bipartite.

Choose a proper two-coloring
\[
 \varepsilon:\mathcal R_{\mathrm{exc}}\longrightarrow\{0,1\},
 \qquad \varepsilon(s)+\varepsilon(t)=1
 \quad\text{for every edge joining }s\text{ and }t.
\]
Define the new residue restrictions by
\[
 \widetilde b_r(y)=
 \begin{cases}
  b_r(y),&r\in\mathcal R_{\mathrm{per}},\\
  \varepsilon(r),&r\in\mathcal R_{\mathrm{exc}},
 \end{cases}
 \qquad y\in\Z^d,
\]
and define $\widetilde b:\Z^d\to\{0,1\}$ by
\[
 \widetilde b(r+qy)=\widetilde b_r(y)
 \qquad(r\in\mathcal R,\ y\in\Z^d).
\]
We verify every tiling equation using its index pair $(i,r)$. If $J_{i,r}=\varnothing$, all its summands are unchanged. If $J_{i,r}=\{f_1,f_2\}$, the summands indexed by $F_i\setminus J_{i,r}$ were identically zero and remain so. Its two remaining summands are now the constants $\varepsilon(s(r,f_1))$ and $\varepsilon(s(r,f_2))$, whose sum is one by the edge associated with $(i,r)$. Hence $\ind_{F_i}*\widetilde b=1$ for every $i$. The same coloring therefore preserves the entire family of tiling equations simultaneously.

Finally, for each $r\in\mathcal R$ define
\[
 A_r=\{r+qy:y\in\Z^d,\ \widetilde b_r(y)=1\}.
\]
If $r\in\mathcal R_{\mathrm{per}}$, choose a rank-$(d-1)$ period subgroup $L_r$ of $b_r$; then $A_r$ is invariant under $qL_r$. If $r\in\mathcal R_{\mathrm{exc}}$, the set $A_r$ is either empty or the whole coset $r+q\Z^d$, and is invariant under $q\Z^d$. The sets $A_r$ are pairwise disjoint, and their union is the support of $\widetilde b$. Thus they give a finite disjoint decomposition into codimension-one-periodic sets of a joint co-tile, as required.
\end{proof}

\begin{proof}[Proof of Theorem~\ref{thm:periodize}]
Apply Theorem~\ref{thm:normal} with $p=1$ to $a=\ind_A$. The selected Boolean configuration $b$ lies in the translation orbit closure of $a$, so it remains a joint co-tile indicator for the entire family. Its residue restrictions satisfy the hypotheses of Proposition~\ref{prop:graph}. That proposition supplies a piecewise $(d-1)$-periodic joint co-tile. Theorem~\ref{thm:MSSreplacement} supplies a fully periodic joint co-tile.
\end{proof}

\begin{proof}[Proof of Theorem~\ref{thm:main} and Corollary~\ref{cor:subfamily}]
Let $A$ be a joint co-tile for the family under consideration. If a tile is a singleton, its equation forces $A=\Z^d$; the remaining equations then force every tile to be a singleton, and the conclusion follows. Otherwise, apply Theorem~\ref{thm:MSSdecomp} to the independent $(d-1)$-subfamily. It gives a bounded real codimension-one decomposition of $\ind_A$. Theorem~\ref{thm:periodize}, applied to all the tiles in the family, gives the required fully periodic joint co-tile.
\end{proof}

The construction uses two distinct operations. Theorem~\ref{thm:normal} first selects a configuration in the translation orbit closure, thereby preserving every closed translation-invariant constraint. Proposition~\ref{prop:graph} then replaces its exceptional residue restrictions by constants while leaving its codimension-one-periodic restrictions unchanged. The latter operation preserves the given equations because every summand is nonnegative and each equation has right-hand side exactly $1$. It need not preserve arbitrary additional constraints. In particular, no periodic replacement is asserted here for convolution systems with signed coefficients or for multiple-tiling equations whose right-hand side is an integer greater than $1$.

\section{Consequences}\label{sec:consequences}

We first recall the independence part of \cite[Theorem 1.9]{MSS24}. The construction preserves a prescribed periodic co-tile.

\begin{lemma}[Completing a periodic co-tile]\label{lem:completion}
Let $\{0\}\subsetneqq F\subset\Z^d$ be finite, and suppose that $F\oplus A=\Z^d$ with $A$ fully periodic. There are $d-1$ further finite sets containing $0$ which, together with $F$, form an independent $d$-tuple and all have $A$ as a co-tile.
\end{lemma}
\begin{proof}
Fix a finite-index period subgroup $L$ of $A$. Distinct points of $F$ lie in distinct cosets of $L$: otherwise their translates of the nonempty set $A$ would coincide, contrary to uniqueness in the tiling.

Suppose an independent tuple of $t<d$ tiles with joint co-tile $A$ has been constructed, beginning with $F$ when $t=1$. Let $\mathcal V$ be the finite collection of rational subspaces spanned by all choices of one nonzero vector from each of these $t$ tiles. Every $V\in\mathcal V$ has dimension $t<d$.

For each $f\in F\setminus\{0\}$, we can choose
\[
 g_f\in(f+L)\setminus\bigcup_{V\in\mathcal V}V.
\]
To justify this avoidance, write $L=B\Z^d$ with an invertible integer matrix $B$. For each $V\in\mathcal V$, choose a nonzero rational linear form $\lambda_V$ vanishing on $V$. The polynomial
\[
 \prod_{V\in\mathcal V}\lambda_V(f+Bz)
\]
is nonzero, since each factor has a nonzero linear part. A nonzero polynomial over $\Q$ cannot vanish on all of $\Z^d$: induction on the number of variables reduces this assertion to the fact that a nonzero univariate polynomial has only finitely many roots. Some integer $z$ therefore gives the required $g_f$.

Set $g_0=0$ and $G=\{g_f:f\in F\}$. Distinctness of the cosets $f+L$ shows that $G$ has $|F|$ elements. Since $g_f+A=f+A$ for each $f\in F$, we have $G\oplus A=\Z^d$. The avoidance condition says exactly that adjoining $G$ preserves independence. Repeating the construction gives an independent tuple of length $d$.
\end{proof}

\begin{proof}[Proof of Corollary~\ref{cor:companions}]
If $F$ has a fully periodic co-tile $A$, apply Lemma~\ref{lem:completion} and retain $d-2$ of the added tiles. The resulting $(d-1)$-tuple is independent and shares the same $A$. Conversely, a joint co-tile for the independent tuple $(G_1,\ldots,G_{d-2},F)$ can be replaced by a fully periodic joint co-tile by Theorem~\ref{thm:main}. This is in particular a fully periodic co-tile for $F$.
\end{proof}

\begin{proof}[Proof of Corollary~\ref{cor:decision}]
The input condition is a finite one. For each candidate $(d-1)$-subfamily, check the rational rank of every choice of one nonzero vector from each tile. If normalizing points have not been specified, there are only finitely many choices of such points to check. All these rank computations use exact rational arithmetic. If any tile is a singleton, its equation forces the joint co-tile to be $\Z^d$; the family has a joint co-tile exactly when every tile is a singleton. We may therefore assume that all tiles are nontrivial.

Run the following two searches in parallel. For the positive search, enumerate $q=1,2,\ldots$ and all Boolean functions
\[
 b:(\Z/q\Z)^d\longrightarrow\{0,1\}.
\]
For each function, check the finite collection of equations
\begin{equation}\label{eq:periodicsearch}
 \sum_{f\in F_i}b(r-f)=1
 \qquad\bigl(1\leq i\leq k,\ r\in(\Z/q\Z)^d\bigr).
\end{equation}
The sum is over the original tile points: if two points of $F_i$ have the same residue, both occurrences are retained. A successful check produces a joint co-tile by taking the inverse image of $\{r:b(r)=1\}$ in $\Z^d$. Conversely, every fully periodic set is invariant under $q\Z^d$ for some $q\geq1$, since every finite-index subgroup contains such a scalar subgroup. By Corollary~\ref{cor:subfamily}, the positive search therefore succeeds whenever a joint co-tile exists.

For the negative search, let
\[
 B_n=[-n,n]^d\cap\Z^d,
 \qquad D_n=\bigcup_{i=1}^k(B_n-F_i).
\]
For each $n$, test all assignments $b:D_n\to\{0,1\}$ against
\begin{equation}\label{eq:obstructionsearch}
 \sum_{f\in F_i}b(x-f)=1
 \qquad(1\leq i\leq k,\ x\in B_n).
\end{equation}
If none satisfies these equations, they form the required finite unsatisfiable collection.

If every such finite system were satisfiable, let $X_n\subset\{0,1\}^{\Z^d}$ be the set of configurations satisfying~\eqref{eq:obstructionsearch}. Each $X_n$ would be nonempty, because an assignment on $D_n$ extends arbitrarily to the rest of the lattice. The sets $X_n$ are closed and nested. Compactness would give a point in $\bigcap_nX_n$, and this point would satisfy every joint tiling equation. Thus the negative search succeeds whenever no joint co-tile exists.

Every individual stage of either search is finite. Alternating their computational steps gives an algorithm which terminates on every input in the asserted class, with exactly one of the two certificates. Both searches and the independence checks are specified uniformly in $d$, so the dimension may be part of the input.
\end{proof}

\begin{remark}
The decision procedure also gives a computable scalar period bound in terms of $d$, the number $k$ of tiles, and a coordinate bound $R$ with $F_i\subset[-R,R]^d$. There are only finitely many families within these bounds. Check which ones contain an independent $(d-1)$-subfamily, run the decision procedure on each, and take the maximum of the periods returned on the positive instances, or $1$ if there are none. This argument gives no explicit quantitative estimate for that bound.
\end{remark}

\section{Examples}\label{sec:examples}

The first two examples separate independence from property $(\star)$.

\begin{example}[Failure of $(\star)$ in every order]\label{ex:coordinate}
For $d\geq3$, set
\[
 F_i=\{0,e_i,2e_i,3e_i\}\quad(1\leq i\leq d-1),
 \qquad
 A=\left\{x\in\Z^d:\sum_{j=1}^dx_j\equiv0\pmod4\right\}.
\]
Every choice of nonzero vectors from the tiles consists of nonzero multiples of distinct coordinate vectors, so the tuple is independent. Under the homomorphism $x\mapsto\sum_jx_j\pmod4$, each tile maps bijectively to $\Z/4\Z$. Hence $A$ is a fully periodic joint co-tile.

Every nonzero choice has the same rational span, namely $\spanq(e_1,\ldots,e_{d-1})$. Changing the vector selected from the first tile, for example from $e_1$ to $2e_1$, preserves this span and violates~\eqref{eq:star}. The same argument applies after any permutation of the tiles, so property $(\star)$ fails in every order.
\end{example}

\begin{example}[Two full-rank tiles in $\Z^3$]\label{ex:fullrank}
Let
\[
 F_1=\{0,e_1,e_2,e_3\},\qquad
 F_2=\{0,e_1+e_2,e_1+2e_2,e_1+e_2+e_3\}.
\]
The determinant of the three displayed nonzero vectors is $1$ for each tile, so both tiles span $\Q^3$. No nonzero point of $F_2$ is parallel to a coordinate axis. Thus the pair is independent.

The pairs $(e_1,e_1+e_2)$ and $(e_2,e_1+2e_2)$ span the same $xy$-plane. Their first entries differ, so $(F_1,F_2)$ fails $(\star)$; their second entries also differ, so $(F_2,F_1)$ fails $(\star)$. The homomorphism
\[
 \chi(x,y,z)=x+2y+3z\pmod4
\]
takes the values $0,1,2,3$ on the displayed points of $F_1$, and $0,3,1,2$ on those of $F_2$. Consequently $\ker\chi$ is a fully periodic joint co-tile.
\end{example}

\begin{example}[A joint co-tile that is not fully periodic]\label{ex:nonperiodic}
For $d\geq2$, let
\[
 F_i=\{0,e_i\}\quad(1\leq i\leq d-1),
 \qquad
 A=\left\{x\in\Z^d:
 \sum_{i=1}^{d-1}x_i\equiv\ind_{\{0\}}(x_d)\pmod2\right\}.
\]
The tiles are independent, and translation by $e_i$ reverses the displayed parity condition. Therefore $F_i\oplus A=\Z^d$ for every $i$.

If $t\in\Z^d$ is a period of $A$, then
\[
 \ind_{\{0\}}(n+t_d)-\ind_{\{0\}}(n)
 \equiv\sum_{i=1}^{d-1}t_i\pmod2
 \qquad(n\in\Z).
\]
For $t_d\ne0$, the left-hand side equals $1$ modulo $2$ at exactly two values of $n$ and is zero elsewhere. It cannot be constant. Thus every period of $A$ has $t_d=0$, and the displayed congruence then requires $\sum_{i=1}^{d-1}t_i$ to be even. Conversely, every vector satisfying these two conditions is a period. The full period group is therefore
\[
 \left\{t\in\Z^d:t_d=0,\quad
     \sum_{i=1}^{d-1}t_i\equiv0\pmod2\right\},
\]
which has rank $d-1$. Hence $A$ is codimension-one periodic but not fully periodic. The same family has the fully periodic joint co-tile
\[
 A'=\left\{x\in\Z^d:\sum_{i=1}^{d-1}x_i\equiv0\pmod2\right\}.
\]
\end{example}

The final example explains the distinction between bounded real decompositions and disjoint periodic partitions at the level of Boolean configurations.

\begin{example}[An exceptional Boolean configuration]\label{ex:carry}
Let $d\geq2$, and choose real numbers $\alpha,\beta,\gamma$ such that $1,\alpha,\beta,\gamma$ are linearly independent over $\Q$. Define
\begin{align*}
 u(x)={}&2-\{\alpha x_1+\beta\}
          -\{\alpha x_2+\gamma\}\\
        &\hspace{18mm}-\{-\alpha(x_1+x_2)-\beta-\gamma\}.
\end{align*}
The three quantities inside the fractional parts sum to zero, and none is an integer. Their fractional parts therefore sum to either $1$ or $2$, so $u$ is Boolean. Each fractional-part summand is invariant under a rank-$(d-1)$ subgroup, with respective defining linear forms $x_1$, $x_2$, and $x_1+x_2$.

These forms are pairwise nonparallel and span a two-dimensional rational space. The slope-intercept pairs are $(\alpha,\beta)$, $(\alpha,\gamma)$, and $(-\alpha,-\beta-\gamma)$; each pair is rationally independent together with $1$. Thus $u$ has the exceptional form to which Lemma~\ref{lem:Elevels} applies. Both level sets are nonempty: the first two phases are jointly equidistributed modulo one, and their fractional parts have sum both less than and greater than $1$. Neither level set contains a nonempty codimension-one-periodic subset, by that lemma. In particular, neither level set is piecewise $(d-1)$-periodic. Hence a bounded real codimension-one decomposition of a Boolean configuration need not be a disjoint periodic partition.
\end{example}

\section*{Declaration of generative AI use}
OpenAI's ChatGPT (Codex) was used in the development and preparation of this manuscript for mathematical exploration, proof drafting and checking, literature searches, and language editing. Responsibility for the results and their presentation rests with the authors.

\end{document}